\documentclass[11pt]{amsart}

\usepackage{amsgen,amsthm,amstext,amsbsy,amsopn,amsfonts,amssymb, enumerate,xcolor}

\usepackage{amsmath}

\newcommand\la{\langle}
\newcommand\ra{\rangle}

\newcommand\hh{{\mathfrak h}}

\newcommand\nn{{\mathfrak n}}

\newcommand\vv{{\mathfrak v}}

\newcommand\zz{{\mathfrak z}}

\newcommand\iso{{\mathfrak{iso}}}

\newcommand\CC{\mathbb C}

\newcommand\RR{\mathbb R}

\newcommand\Auto{\operatorname{Auto}}

\newcommand\Iso{\operatorname{Iso}}

\DeclareMathOperator{\En}{En}

\DeclareMathOperator{\arccosh}{arccosh}
\DeclareMathOperator{\sech}{sech}

\theoremstyle{plain}

\newtheorem{thm}{Theorem}[section]
\newtheorem{lem}[thm]{Lemma}
\newtheorem{prop}[thm]{Proposition}
\newtheorem{cor}[thm]{Corollary}

\theoremstyle{definition}

\newtheorem{rem}[thm]{Remark}
\newtheorem{example}[thm]{Example}

\newcounter{casenum}
\newenvironment{caseof}{\setcounter{casenum}{1}}{\vskip.5\baselineskip}
\newcommand{\case}[2]{\vskip.5\baselineskip\par\noindent {\bfseries Case \arabic{casenum}:} #1\\#2\addtocounter{casenum}{1}}

\begin{document}
	
	\title[Magnetic equations on the Heisenberg group]
	{Magnetic equations  on the Heisenberg group: symmetries, solutions and the inverse problem of the calculus of variations}

	\author{Gabriela P. Ovando, Mauro Subils}
	
	\thanks{{\it (2000) Mathematics Subject Classification}: 53C99, 70G65, 70F17, 22E25}
	
	
	\thanks{Partially supported by  ANPCyT, SCyT (UNR) and CONICET}
	
	\address{ Departamento de Matem\'atica, ECEN - FCEIA, Universidad Nacional de Rosario.   Pellegrini 250, 2000 Rosario, Santa Fe, Argentina.}
	
	\
	
	\email{gabriela@fceia.unr.edu.ar}
	
	\email{subils@fceia.unr.edu.ar}
	

	\begin{abstract}
	The Heisenberg Lie group $H_3$ is modeled on the differentiable structure of $\RR^3$ but equipped with  another non-commutative  product operation. 	By fixing the usual metric on the Heisenberg Lie group, this work provides a comprehensive overview of the behavior of magnetic geodesics for any invariant Lorentz force.  After writing the magnetic equations, we found symmetries that enable the explicit computation of  the magnetic trajectories  for any homogeneous exact and non-exact magnetic form. Finally we show that these magnetic
	trajectories are solutions of a variational problem: we present explicit examples of Lagrangians. 
	\end{abstract}
	
	\maketitle
	
	\setcounter{tocdepth}{1}
	\tableofcontents

	\noindent\section{Introduction}
	“The Heisenberg group was born long before it was christened.” This is the opening sentence of the book about Harmonic Analysis written by G. Folland in 1989 \cite{Fo}. As he said, the object with the older history is
	the Heisenberg Lie algebra $\hh_{n}$; it appeared  in the Poisson 	bracket relations of canonical coordinates in Hamiltonian mechanics (Jacobi,
	1843) and then in the canonical commutation relations of quantum 	mechanics (Born and Jordan, 1925).  See references in the very nice survey \cite{Fo2}. Indeed, the Heisenberg group is related to a wide spectrum of mathematical phenomena, in harmonic analysis, subriemannian geometry and  in the mathematical theory of signal and 	image processing, to mention some diverse examples.  
	
		The Heisenberg group $H_3$ has dimension three and can be seen as a  group  acting transitively by isometries on $\RR^3$ whenever we change the Euclidean metric by  the metric given in usual coordinates at $p=(x,y,z)$ by:
	$$g=(1+\frac{y^2}4)  dx^2 + (1 + \frac{x^2}4) dy^2 + dz^2 + \frac{y}2 dx dz - \frac{x}2 dy dz.$$
	
	The action of the element $(a,b,c)\in H_3$ on $(x,y,z)\in \RR^3$ is given by the  product:
\begin{equation}\label{operation}(a,b,c)(x,y,z)= (x+a, y+b, z+c+\frac12 (ay-by)).
\end{equation}
	
	Since the action is transitive, $H_3$ is seen as a homogeneous manifold. In fact, it is  the low-dimensional (non-abelian) 2-step nilpotent Lie group. If we concentrate the attention to geometrical features, we note that the geodesic equations on $H_3$ were solved by Kaplan in \cite{Ka}, equations which are still true for 2-step nilmanifolds.   In the more general  context  of 2-step nilmanifolds, there are  many works on  geodesics and their flows studied by several authors, such as  \cite{dC,dM,Eb, KOS, LP,Mar, Ma}. These works make use of techniques of  Lie theory.  

	In the present work we focus on magnetic trajectories, also called   magnetic geodesics, on  the Heisenberg Lie group $H_3$. We consider magnetic 2-forms, $\omega_F$ which are invariant by the action above. Recall that magnetic geodesics are solutions of a magnetic equation
	\begin{equation}\label{mageq}
		\nabla_{\gamma'}{\gamma'}= q F\gamma'
	\end{equation}
	where    $\nabla$ is the corresponding Levi-Civita connection and $F$ is a skew-symmetric $(1,1)$-tensor such that  the corresponding 2-form $\omega_F= g(F\cdot ,\cdot)$ is closed. In particular, geodesics occur whenever $F=0$. Magnetic geodesics are known as {\em flowlines} of the dynamical system associated with the Lorentz force $F$. Equation \eqref{mageq} is also called a {\em Landau-Hall} differential equation in \cite{BC} and it is a  generalization of the Landau-Hall problem, that studies the motion  of a charged particle in the presence of a static magnetic field.  Examples of the equation above arise from the K\"ahler form \cite{Ad}, or from the contact magnetic field  on Sasakian manifolds \cite{CF}. 
	

	 Thus, the first step is to write down the magnetic equations on $H_3$ for any invariant Lorentz force $F$. In coordinates, the magnetic equations yield a 
	second order differential system of the form:
	\begin{equation}\label{douglaseq}
		y_i''=F_i(t,y_i,y_i') \qquad \mbox{ for all } i=1, 2,3,
	\end{equation}
	for some polynomial functions $F_i$, see Equation \eqref{magnetic-Heis}.


Two main goals inspire the present work:
\begin{enumerate}[(i)]
	 \item to solve the magnetic equations \eqref{mageq} for any invariant force $F$;
	\item to address the inverse problem of the calculus of variations:  to find a Lagrangian $L$ such that the corresponding Euler-Lagrange equations \eqref{douglaseq} are equivalent to  the magnetic equations. 
		 	
\end{enumerate}

Existence results of   magnetic trajectories have been  obtained using various methods, such as   heat flow techniques in  \cite{BH},  methods from dynamical systems and symplectic geometry, see for instance \cite{Ar, BP,Gi},  variational methods for multivalued functionals (Morse-Novikov-theory)\cite{NT, Tai}, tools from  Lie theory  \cite{BJ1,BJ2} and studies on specific manifolds \cite{EI,Sc}.

 A classical topic of research in geometry is to determine when two different metrics on the same manifold share the same geodesics.  For magnetic trajectories, the question was studied in \cite{BM}. Clearly, magnetic geodesics of $(g, \Omega)$ coincide up to reparametrization with those of the rescaled system $(\alpha g, \beta \Omega)$ for constants $\alpha>0, \beta \neq 0$.  A key difficulty in the case of magnetic geodesics, is that dilations on the parameter do not yield a solution of the corresponding magnetic equation. On the other hand,  the Lie group $H_3$ has a unique left-invariant metric up to homothety \cite{La}.  In the present paper we work two possible generalizations of this question. By fixing the metric, we  associate magnetic fields and trajectories under an equivalence relation. 
This is done under group's actions and it is our proposal for solving the magnetic equations for any invariant magnetic field. On the other side we determine which magnetic trajectories are geodesics. 

Let $\Iso(M)$ denotes the isometry group of the Riemannian manifold $M$, then $\Iso(M)\times \RR^*$ acts on the set of Lorentz forces and also, on the set of curves. Moreover, if $\gamma$ is a magnetic trajectory for the Lorentz force $F$, then $(\psi,r)\cdot\gamma$ is a magnetic trajectory for  $(\psi,r)\cdot F$, where $(\psi,r)\in \Iso(M)\times \RR^*$, see Lemma \ref{lem2}. A consequence is that the isotropy subgroup of this action gives symmetries of the magnetic equations, once one fixes the Lorentz force $F$. 

The Heisenberg Lie algebra $\hh_3$ has a basis $e_1,e_2,e_3$ in which  any Lorentz force is represented by a skew-symmetric matrix 
that, under the action above is equivalent to exactly one of the following matrices:
$$F=0,\qquad \quad a_{\rho}=\left( \begin{matrix}
	0 & -\rho  & -1\\
	\rho & 0 &0 \\
	1 & 0 & 0
\end{matrix}
\right), \quad \mbox{ for } \rho\geq 0\quad \mbox{ or }b=\left( \begin{matrix}
	0 & -1& 0\\
	1 & 0 &0 \\
	0 & 0 & 0
\end{matrix}
\right). 
$$
Solutions for $F=0$ are geodesics. For the Lorentz force $b$,  
 the magnetic curves were found in \cite{EGM, MN}, and the case $a_0$ was solved in \cite{OS}. This last case corresponds to harmonic 2-forms. Indeed we compute the solutions for non-equivalent forces.
 

\smallskip

{\bf Main results}

\begin{itemize}
	\item Theorem \ref{thmsolutionmagHeisII} in the present work  determines the solutions to question (i) as  $\gamma(t)=\exp(x(t)e_1+y(t)e_2+z(t)e_3)$ of the magnetic equation passing through the identity element  for any initial condition $\gamma'(0)=x_0 e_1+y_0 e_2 +z_0e_3$. Such curves are in one-to-one correspondence to functions
	$$	x:\RR\to\RR \quad \mbox{solutions of Equation} \quad x''(t)+h'(x(t)) h(x(t))=\rho, \, x(0)=0 \mbox{ and } x'(0)=x_0, $$
	where $h(x)=  \frac{x^2}{2}+(z_0+\rho)x+y_0+1$. In Table \ref{table:1} we give all non-trivial solutions $x(t)$. 
		The case $b$ above, corresponding to left-invariant exact  2-forms is special and isolated. 
		
	Later we  show that the only magnetic  geodesics which are geodesics are certain one-parameter subgroups, see Corollary \ref{mageo}. 
	
	\item Theorem \ref{Lagrangian}  solves question (ii) on the Heisenberg  group $H_3$. We provide  conditions for  a given function to become a  Lagrangian and propose an explicit example. The key insight is that any left-invariant closed 2-form on $H_3$ is exact, see Lemma \ref{closed2areexact}. 
\end{itemize}

The inverse problem of the calculus of variations requires a work backward from the  differential magnetic equations on $H_3$ to find a corresponding Lagrangian.  A first approach  is the use of variational integrating factors, a  procedure outlined by Douglas for two variables in \cite{Do}. 
More generally,  Helmholtz  conditions \cite{He}  assert the existence of a local  Lagrangian.   Applications of the inverse problem and the Helmholtz conditions include the 
control and feedback stabilization of mechanical systems, see the first chapters of  \cite{Ze}. On the other side, differential equations on the Heisenberg group is an active research field as it is shown in \cite{PP}. 

Here, we make  use of   tools such as those in \cite{Cretal} involving differentials of forms and vector fields on  manifolds \cite{Br}, that   provide a global Lagrangian. In fact, by writing the   Lagrangian as
$$L = \En + \theta$$
where $\En$ denotes
 the energy, we are able to give explicit 1-forms  $\theta$ on $TH_3$. In this way the magnetic trajectories obtained above  can be identified with the totality of extremals of a variational problem. To know more recent developments, see for instance \cite{DP}.

 The paper is organized as follows: Section 2 contains the  basis for the actions explained above, which are given on any Riemannian manifold. We present also basic features of the Heisenberg Lie group and its Lie algebra. Section 3 is devoted to the computations of the solutions of the magnetic geodesics on $H_3$ and in Section 4 we deal with the Inverse problem of the calculus of variations. 
 
 \smallskip 

The authors kindly thank Prof. R. Bryant for the generous suggestion to study question (ii) and for providing references to address it.

\section{Symmetries for the magnetic equation}	\label{symmetries}
In this section we find symmetries of the magnetic equation. This is achieved by considering an action on the set of Lorentz forces and on the set of differentiable curves of a Riemannian manifold.

Let $(M,g)$  denote a Riemannian manifold with Levi-Civita connection $\nabla$. Let $\Omega$ be a differentiable 2-form on $M$, then there exists a unique   skew-symmetric $(1,1)$-tensor $F:TM \to TM$  defined by the relation:
$$\Omega_F(U,V)=g(FU,V), \quad \mbox{ for all }\quad  U,V\in \chi(M).$$
Conversely, the relation above defines a 2-form whenever $F$ is given. If the 2-form is closed, then one gets an extra condition on $F$. In such case the tensor $F$ is called a  {\em  Lorentz force}. 

A differentiable {\em curve} or {\em trajectory} on $M$ is a differentiable function $\gamma:I\to M$ defined on some open interval such that $0\in I\subset\mathbb{R}$. A {\em magnetic trajectory} is a curve $\gamma$ satisfying Equation \eqref{mageq} for some Lorentz force $F$. 

Given a Lorentz force $F$ and a magnetic trajectory $\gamma$, one has
$$\frac{d}{dt}g(\gamma'(t), \gamma'(t))=2g(\nabla_{\gamma'(t)}\gamma'(t), \gamma'(t))=2g(F\gamma'(t), \gamma'(t))=0, $$
implying that magnetic  curves have  constant velocity.  


Note that a reparametrization of a magnetic curve could not be a solution of Equation \eqref{mageq}. In fact, take the curve $\tau(t)=\gamma(r t)$ with $r\neq 0,\,1$. Then it holds $\tau'(t)=r \gamma'(r t)$,  so that one has 
$$\nabla_{\tau'(t)}\tau'(t)= r^2 F\gamma'(r t),\mbox{ while on the other side } F \tau'(t)=r F \gamma'(r t).$$	
Denote by  $\mathcal{F}$ the set of Lorentz forces on $M$ and by $\mathcal{C}$ the set of all differentiable curves on $M$. 
Let  $\Iso(M)$ be the isometry Lie group of $M$ and let  $\mathbb{R}^*$ be the multiplicative group. 

There is a left-action of the product group  $\Iso(M) \times \RR^*$ on $\mathcal{F}$ given by:
\begin{equation}\label{action1}
	(\psi, r)\cdot F =r \psi_*\circ  F \circ  \psi^{-1}_*, \quad \mbox{for } \psi \in \Iso(M), r\in \RR^*, F\in \mathcal F, 
\end{equation}
where $\psi_*$ denotes the differential of $\psi$. 
The same group acts on the set $\mathcal{C}$ in the following way:
\begin{equation*}
	((\psi, r)\cdot \gamma)(t)=(\psi\circ\gamma)(rt), \quad \mbox{for } \psi \in \Iso(M),  r\in \RR^*, \gamma \in \mathcal C. 
\end{equation*} 	

Note that  the associated $2$-form of the $(1,1)$-tensor $(\psi, r)\cdot F$ is a multiple of the pullback by $\psi^{-1}$ of $\Omega_F$. This shows that $(\psi, r)\cdot F$ is a Lorentz force:  $(\psi, r)\cdot F\in\mathcal{F}$. It is easy to check that these left-actions are well defined. For instance take   $(\psi,r)$, $(\phi,s)\in\Iso(M)\times\mathbb{R}^*$ and $\gamma\in\mathcal{C}$ then one has 
$$((\psi,r)\cdot((\phi,s)\cdot\gamma))(t)=\psi\circ((\phi,s)\cdot\gamma))(rt) =(\psi\circ\phi\circ\gamma)(srt)=((\psi\circ\phi,rs)\cdot\gamma)(t). $$
Once the left-action of a group $H$ on a set $S$ is defined, recall that the {\em isotropy subgroup} for $s\in S$ is the subgroup of $H$ given by $$H_s=\{h\in H \,:\, h\cdot s=s\}, $$
while the orbit of $s\in S$ is the subset  defined by $H\cdot s=\{h\cdot s \, : \, h\in H\}$. 

\begin{lem} \label{lem2} Let $(M,g)$ be a Riemannian manifold, and let $F$ denote a Lorentz force on $M$. Consider
	the group $H=\Iso(M)\times\mathbb{R}^{*}$. 
	\begin{enumerate}[(i)]
		\item If $\gamma$ is a magnetic trajectory for the Lorentz force $F$ then for any $(\psi,r)\in\Iso(M)\times\mathbb{R}^{*}$, the curve $(\psi,r)\cdot\gamma$ is a magnetic trajectory for the Lorentz force $(\psi,r)\cdot F$. 
		\item  Let  $H_F$ be the isotropy subgroup for the Lorentz force $F$. Then $H_F$ is a group of symmetries for  the magnetic equation $\eqref{mageq}$. 
		
		Furthermore,  $H_F$ consists of elements $(\psi,r)$, with $r=\pm 1$ and  such that
		\begin{equation}\label{symmetry}\psi_* \circ F \circ \psi_*^{-1}=F\quad \mbox{ or } \quad \psi_* \circ F \circ \psi_*^{-1}= -F.
		\end{equation}
	\end{enumerate}
\end{lem}
The proof of (i) in the  Lemma follows from usual computations. 	The second statement is a consequence from the definitions.

To complete the proof, let $F$ be a Lorentz force of $M$ and let $(\psi,r)\in H_F$. As mentioned above $F_p : T_p M \to T_p M$ is a non-zero skew-symmetric map at every tangent space for any $p\in M$. As usual denote by  $\Vert v\Vert$  the norm of $v\in T_pM$:  $\Vert v\Vert=g_p(v,v)^{1/2}$.  It follows that  if  $(\psi, r)$  is an element of  the isotropy subgroup for $F\neq 0$,  it must hold
\begin{equation*}
	0< \sup_{\Vert v\Vert=1}\Vert F_p(v)\Vert = \sup_{\Vert v\Vert=1}|r| \, \Vert \psi_* F_p \psi^{-1}_*(v)\Vert= |r|\sup_{\Vert v\Vert=1} \Vert F_p \psi^{-1}_*(v)\Vert = |r| \sup_{\Vert v\Vert=1}\Vert F_p (v)\Vert,  
\end{equation*}
which implies  $r=1$ or $r=-1$. This says that $(\psi,r)\in H_F$ for a Lorentz force $F$ if and only if Equations  \eqref{symmetry} hold.



Next we concentrate on nilpotent Lie groups. Recall that they correspond to  nilpotent Lie algebras. 

Denote by $\nn$ a Lie algebra, and consider the lower central series defined as follows
$$\nn_0=\nn, \quad \nn_j=[\nn, \nn_{j-1}]\quad \mbox{ for }j\geq 1.$$
The Lie algebra $\nn$ is called nilpotent if there exists $k>0$ such that $\nn_k=0$. The smallest $k$  is called the step of nilpotency for  the nilpotent Lie algebra. We assume here that  $\nn$ is non-abelian. 

Indeed,  the simpler algebraic examples of nilpotent Lie algebras are the 2-step ones verifying $[U,[V,W]]=0$ for all $U,V,W\in \nn$. 

\begin{example}
	The Heisenberg Lie algebra $\hh_3$ is a 2-step nilpotent Lie algebra. It has dimension three and it is spanned by vectors $e_1, e_2, e_3$ satisfying the non-trivial Lie bracket relation
	$$[e_1,e_2]=e_3.$$
	The corresponding Lie group is called the Heisenberg Lie group, denoted by $H_3$, which is modeled on $\RR^3$ endowed with the  product operation of Equation \eqref{operation}. 
\end{example}

Assume now that the manifold is a nilpotent Lie group equipped with a metric $(N,\la\,,\,\ra)$ for which translations on the left by elements of the group are isometries,  called a left-invariant metric.
In this context, the isometries that are compatible with the underlying algebraic structure are the isometric automorphisms, usually called orthogonal automorphisms. The group of such automorphisms is denoted by $\Auto(N)$.

In \cite{Wo} Wolf  proved that the group of isometries of $(N,\la\,,\,\ra)$ is given by $\Iso(N,\la\,,\,\ra)=N.\Auto(N)$. Thus, 
any isometry can be written as
$$L_p\circ \psi, \quad \mbox{ for } \psi\in \Auto(N), p\in N,$$
where $L_p$ denotes a translation on the left by $p\in N$. 

Indeed the subgroup of translations on the left $\{L_p\}_{p\in N}$ is isomorphic to the Lie group $N$. Moreover it is a normal subgroup since it holds
$$L_p \circ \psi=\psi \circ L_{\psi(p)}\quad \mbox{ for any } \psi\in \Auto(N), p\in N.$$	 
Consider a left-invariant Lorentz force, that means that it is invariant by translations on the left:
$$L_p \circ F= F\circ L_p, \quad \mbox{ for any } p\in N. $$
In terms of the action of Lemma \ref{lem2} one gets,   $L_p\cdot F=F$, saying that the subgroup of left-translations $N$ is contained in the isotropy subgroup $H_F$.

Let $\mathcal F_l$ denote the set of left-invariant Lorentz forces. Thus  the action of the group $\Iso(N)\times \RR^*$ on the set $\mathcal F_l$ is determined by the action of orthogonal automorphisms:
$$(L_p\circ \psi) \cdot F=\psi\cdot  F.$$
In particular from the analysis above one obtains the next result. 
\begin{cor} Let $F$ be a left-invariant Lorentz force on the nilpotent Lie group $(N, \la\,,\, \ra)$ equipped with a left-invariant metric. Then for the isometry $L_p\circ \psi$, the pair $(L_p\circ \psi, \varepsilon)$, with $\varepsilon =\pm 1$,  belongs to the isotropy of $F$  if and only if 
	$$\psi_*\circ F\circ \psi_*^{-1}=\varepsilon F,$$
	with $\psi \in \Auto(N)$.
\end{cor}


\begin{example}\label{ortAutHeis} Take the Heisenberg Lie algebra equipped with its canonical metric: the metric for which the basis $e_1,e_2,e_3$ is orthonormal. The center is spanned by $e_3$ and its orthogonal complement $\vv$ is spanned by $e_1, e_2$. Thus one has
	$$\hh_3=\vv \oplus \zz, \quad \mbox{ with }\quad \zz=\RR e_3, \vv=\zz^\perp.$$
	
	A classical result of Lie theory states that a automorphism $\psi$ of the Lie algebra $\hh_3$ in the connected component of the identity, is the exponential of a derivation: $d:\hh_3\to \hh_3$: that is,  $d[u,w]=[du,w]+[u,dw]$. Easy computations show that any skew-symmetric derivation of $\hh_3$ has a matrix in the basis $e_1, e_2, e_3$ of the form
	$$\left( \begin{matrix}
		0 & -a & 0\\
		a & 0 & 0\\
		0 & 0 & 0
	\end{matrix}\right), \qquad \quad a\in \RR.
	$$
	
	Thus, the isometry Lie algebra is $\iso(\hh_3)= \hh_3\oplus \RR$, where any  $a \in \RR$ corresponds to a skew-symmetric derivation as above. Furthermore, the Lie algebra $\iso(\hh_3)$ is isomorphic to the oscillator Lie algebra of dimension four.  The isometry group of the Heisenberg Lie group $(H_3, \la\,,\,\ra)$ is  $\Iso(H_3)= H_3\rtimes\mathrm O(2)$, where the action of $\mathrm O(2)$ on $\hh_3$ is explicitly given by
	$$B\cdot (V,Z) = (B(V),\,det(B)Z), \ \  \mbox{ for }B\in \mathrm O(2), V\in \vv, Z\in \zz.$$
	Note that we identify $\RR^2$ with the subspace $\vv=span\{e_1,e_2\}$. 
	
	And since $H_3$ is simply connected, one makes no distinction between automorphisms of $N$ and $\nn$.
\end{example}

It is easy to prove that  any skew-symmetric map  $F:\hh_3 \to \hh_3$ gives rise to a left-invariant closed 2-form on the Heisenberg Lie group $H_3$. This Lorentz force admits the following matrix representation with respect to the basis $e_1, e_2, e_3$  	
\begin{equation} \label{matrixF}
	F=\left( \begin{matrix}
		0 & -\rho  & - \beta\\
		\rho & 0 & -\alpha \\
		\beta & \alpha & 0
	\end{matrix}
	\right), \quad \alpha, \beta \in \RR. 
\end{equation}

One denotes $F$ by $F_{U,\rho}$ with  $U=\beta e_1+ \alpha e_2\in\hh_3$. 

Now, consider the action of $\Iso(H_3)\times \RR^*$ on the set of left-invariant Lorentz forces $\mathcal F_l$ described in   Lemma \ref{lem2}. This action can be considered equivalently on 2-forms. 

From the definitions, the action of $\Iso(H_3)\times \RR^*$ on the Lorentz force $F$ is determined by $B\in \mathrm O(2)$ and  $r\in \RR^*$ and it holds
$$(B,r)\cdot F_{U, \rho} = r\det(B)F_{BU,\rho}.$$


As above, we identify $\RR^2$ with $\vv=span\{e_1, e_2\}$. Explicitly this identification makes  $e_1 \leftrightarrow (1,0)$ and  $e_2 \leftrightarrow (0,1)$.  

To describe the orbits, take a non-trivial left-invariant Lorentz force $F_{U,\rho}$.

Firstly, note that 
$$(-Id,-1)\cdot F_{U, \rho} = - F_{-U,\rho}=F_{U,-\rho}.$$ 
Thus, one may consider $\rho\geq 0$. 
\begin{itemize}
	\item Assume $U\neq 0$ and $\rho\geq 0$. Thus, there exists a rotation  $B\in \mathrm{SO}(2)$ such that $BU=\Vert U \Vert e_1$. By choosing $r=\Vert U \Vert^{-1}$ one has $(B,r)\cdot F_{U,\rho}=F_{e_1, \widetilde{\rho}}$  for $\widetilde{\rho}=\Vert U\Vert^{-1} \rho$.
	
	
	\item Assume $U=0$. In this case it may hold $\rho\neq 0$ and clearly one can choose $r$ so that $(0,r)\cdot F_{0,\rho}=F_{0,1}$. 
\end{itemize}

Elements of the form $F_{e_1, \rho}$ and $F_{e_1, \widetilde{\rho}}$ are in different orbits for  $\rho\neq \widetilde{\rho}$ and both $\rho, \widetilde{\rho}\geq 0$. In fact, by contradiction assume  that that there exists $(B,r)\in \Iso(H_3)\times \RR^*$ such that $(B,r)\cdot F_{e_1, \rho} = F_{e_1, \widetilde{\rho}}$. 
 This implies  $r \det(B)Be_1=e_1$ and $r\rho=\widetilde{\rho}$. 

Therefore, for $\rho=0$, one has $\widetilde{\rho}=0$.
For $\rho\neq 0$,  one has   $(\widetilde{\rho}/\rho) \det(B)Be_1=e_1$. Clearly $|\widetilde{\rho}/\rho| =1$, which says $\widetilde{\rho}=\rho$ since both are positive. 

For $U=0$ the action $(B,r)\cdot F_{0,1}= F_{0, r}$ showing that this orbit is different from the previous families. 

To compute the isotropy subgroups, we solve the following equations
$$(B,1)\cdot F= F\qquad \mbox{ or } \qquad (B,-1)\cdot F= -F.$$

Recall that $\mathrm{O}(2)$ has two connected components: $\mathrm{SO}(2)$ which consists of matrices with determinant equals one and $\mathrm{SO}^-(2)$ consisting of matrices with determinant minus one. 

\begin{prop} \label{orbits} The orbits under the action of $G=\Iso(H_3)\times \RR^*$ on the set of non-trivial left-invariant Lorentz forces $\mathcal F_l$ are mostly parametrized by $\rho\geq 0$; that is, every point in the following set gives different orbits:
	$$\{F_{e_1, \rho}\}_{\{ \rho \geq 0\}}\cup \{F_{0,1}\}.$$
	While isotropy subgroups are
	\begin{enumerate}[(i)]
		\item For $F_{e_1, \rho}$ with $\rho> 0$: $H_3\rtimes \left( \{(Id,1)\} \cup \{(S,-1)\}\right)$;
		\item For $F_{e_1,0}$: $H_3\rtimes \left( \{(Id,1)\} \cup \{(-S,1)\}\cup \{(-Id,-1)\} \cup \{(S,-1)\}\right)$;
		\item For $F_{0,1}$ one has $H_3\rtimes \left( \mathrm{SO}(2)\times \{1\} \cup \mathrm{SO}^-(2)\times \{-1\}\right),$
	\end{enumerate}
	where $S:\vv\to \vv$ is the linear map determined by the values $S(e_1)=-e_1, S(e_2)=e_2$.
\end{prop}

To compute the isotropy subgroups, we may find $B:\hh_3\to \hh_3$ such that $(B,1)\cdot F_{e_1,\rho}=F_{e_1,\rho}$ or $(B,-1)\cdot F_{e_1,\rho}=-F_{e_1,\rho}$. 
Analogously for $F_{0,1}$. Explicit computations of these relations, prove the statements above finishing the proof.  Note that we consider the trivial extension of the map $S:\vv\to\vv$ to $\hh_3$. 

\medskip

Indeed, one  could choose  another fixed element $U\in \vv$ instead of $e_1$ to parametrize the orbits.

\begin{rem}	Among the classes enumerated in Proposition \ref{orbits} the classes (ii) and (iii) are distinguished. In fact, the class in (ii) corresponds to harmonic left-invariant forms (see definitions for instance in \cite{Wa}), while the class (iii) corresponds to exact  left-invariant forms.
\end{rem}

\section{The magnetic equation on the Heisenberg Lie group}

In this section we write down the equations of magnetic trajectories passing through the identity element on $H_3$.


Let  $\hh_3$ be the Heisenberg Lie algebra of dimension three equipped with  the canonical  metric of Example \ref{ortAutHeis}, which is extended to the Heisenberg Lie group as a left-invariant metric on $H_3$. 

 As mentioned in Example \ref{ortAutHeis} for the Heisenberg Lie algebra $\hh_3$, any  2-step nilpotent Lie algebra ($\nn,\la\,,\,\ra$) has a orthogonal decomposition as vector spaces
\begin{equation}\label{decomp}\nn=\vv \oplus \zz, \quad \mbox{ with }\vv=\zz^{\perp}\mbox{ for } \zz \mbox{ the center of }  \nn.\end{equation}
The Lie bracket on $\nn$ determines the skew-symmetric linear maps $j(Z):\vv\to \vv$ implicitly defined by 
$$\la j(Z)W, V\ra =\la Z, [W,V]\ra\quad \mbox{ for all } W,V\in \vv, Z\in \zz.$$

Let $F$ denote a left-invariant Lorentz force on a 2-step nilpotent Lie group  $(N,\la\,,\,\ra)$ equipped with a left invariant metric. Thus $F$ can be identified with  a  skew-symmetric map $F:\nn\to \nn$ satisfying 
$$\la [X,V], FW\ra + \la [V,W],FX\ra +\la [W,X],FV\ra=0\quad \mbox{ for all } X,V,W\in \nn,$$
which derives from the closedness condition. Let $\pi_{\zz}:\nn\to \zz$ and $\pi_{\vv}:\nn\to \vv$ denote the orthogonal projections with respect to the decomposition above in Equation \eqref{decomp}. Write $F_{\vv}=\pi_{\vv}\circ F$ and $F_{\zz}=\pi_{\zz}\circ F$. In \cite{OS2} the authors considered a decomposition of $F$ as $F=F_1+F_2$ where
$$F_1(V+Z)=F_{\vv}(V)+F_{\zz}(Z) \qquad F_2(V+Z)=F_{\vv}(Z)+F_{\zz}(V) \quad \mbox{ for } V\in \vv, Z\in \zz,$$
and they named a Lorentz force of type I if $F\equiv F_1$ and of type II if $F\equiv F_2$. 

In \cite{OS} the authors derived the equations for the magnetic trajectories on 2-step nilpotent Lie groups equipped with left-invariant metrics. They are ordinary differential equations as we see below.

\begin{lem} \cite{OS} Let $\gamma:I \to N$ be a curve on $(N,\la\,,\,\ra)$ given as $\gamma(t)=\exp(V(t)+Z(t))$, where $\exp:\nn \to N$ denote
	the usual exponential map and $V(t)\in \vv, Z(t)\in \zz$, see Equation \eqref{decomp}. Then $\gamma$ is a magnetic trajectory with $\gamma(0)=e$ for the left-invariant Lorentz force $F$ if and only if  $V(t)$ and $Z(t)$ satisfy the following equations:
	\begin{equation}\label{magnetic-2step}
		\left\{ \begin{array}{rcl}
			V''(t)-j(Z'(t)+\frac12 [V'(t), V(t)])V'(t) & = & q F_{\vv}(V'(t)+ Z'(t)+ \frac12 
			[V'(t),V(t)])\\
			Z''(t)+\frac12 [V''(t),V(t)] & = & qF_{\zz}(V'(t)+ Z'(t)+ \frac12 
			[V'(t),V(t)]).
		\end{array} \right.
	\end{equation}
\end{lem}

In the next paragraphs we shall write these equations explicitly for the Heisenberg Lie group of dimension three. In fact,  $H_3$ can be modeled on $\RR^3$ with the product operation given by
$$(v_1, z_1)(v_2, z_2)=(v_1+v_2, z_1+z_2+\frac12 \la Jv_1,v_2\ra), $$
where $v_i=(x_i,y_i)$ for $i=1,2$ and $J(x,y)=(-y,x)$. And since $H_3$ is simply connected the exponential map is a diffeomorphism. Moreover the exponential map will send $\exp: x e_1 + y e_2 + z e_3 \to (x,y,z)$ in usual coordinates of $\RR^3$.

\begin{example} \label{exacsolutions} 
	Consider the left-invariant Lorentz force $F_{0,\rho }$ on  the Heisenberg Lie group $(H_3,\la\,,\,\ra)$, which corresponds to a left-invariant exact magnetic field for any $\rho>0$. 
	
	Let $\gamma(t)=\exp(x(t)e_1+ y(t) e_2 + z(t)e_3)$ denote a magnetic trajectory passing through the identity element. Assume $\gamma'(0)=x_0e_1+y_0 e_2+z_0 e_3$. Then the system  \eqref{magnetic-2step} above translates to the system:	
	\begin{equation}\label{magnetic-exact}
		\left\{ \begin{array}{rcl}
			x''(t)+(z_0+\rho) y'(t) & = &  0\\
			y''(t)-(z_0+\rho) x'(t)  & = & 0 \\
			z'(t)+\frac{1}{2}(x'(t)y(t)-x(t)y'(t)) & = & z_0
		\end{array} \right.
	\end{equation}
	
	In \cite{EGM} the authors obtain the explicit solutions as:
	\begin{itemize}
		\item if $z_0+\rho\neq 0$, the solution is
		$$\left( \begin{matrix}
			x(t)\\
			y(t)
		\end{matrix}\right) = \frac{1}{z_0+\rho}\left( \begin{matrix}
			\sin(t(z_0+\rho)) & -1+\cos(t(z_0+\rho))\\
			1-\cos(t(z_0+\rho)) & \sin(t(z_0+\rho))
		\end{matrix}
		\right)\left( \begin{matrix} x_0\\ y_0 	\end{matrix}
		\right),$$
		and for $V_0=x_0e_1+y_0e_2$ set
		$$z(t)=\left(z_0+\frac{||V_0||^2}{2(z_0+\rho)}\right)t - \frac{||V_0||^2}{2(z_0+\rho)^2}\sin(t(z_0+\rho)).$$
		\item If $z_0=-\rho$ the solution is 
		$\gamma(t)=\exp(t(x_0e_1+y_0 e_2+ z_0 e_3))$. 
		
	\end{itemize}
\end{example}


\medskip

The aim now is to write the magnetic equations on the Heisenberg Lie group $H_3$ for a left-invariant Lorentz force $F_{U,\rho}$ with $U\neq 0$. As above we write $U=\beta e_1+\alpha e_2\in\vv$. 

Let $\gamma(t)= \exp(V(t)+Z(t))$  denote a  magnetic trajectory on the Heisenberg group passing through the identity element. Writing $V(t)=x(t)e_1+y(t)e_2$ and $Z(t)=z(t)e_3$, the magnetic equations in \eqref{magnetic-2step} for any  Lorentz force follows 
\begin{equation}\label{magnetic-Heis}
	\left\{ \begin{array}{rcl}
		x''(t)+\left(z'(t)+\frac12(y'(t) x(t) -x'(t) y(t)+\rho\right) (y'(t)+\beta) & = &  \rho \beta \\
		y''(t)-\left(z'(t)+\frac12(y'(t) x(t) -x'(t) y(t)+\rho\right) (x'(t)- \alpha)  & = & \rho \alpha \\
		z''(t)+\frac{1}{2}(x''(t)y(t)-x(t)y''(t)) & = & \beta x'(t)+\alpha y'(t)\\
	\end{array} \right.
\end{equation}
with the initial conditions $V(0)+Z(0)=0$, $V'(0)=V_0=x_0 e_1+ y_0 e_2$ and $Z'(0)= z_0e_3$. 





As proved in Proposition \ref{orbits}, for a magnetic trajectory $\gamma$  there is a  symmetry  $\psi=(g,\phi)\in N \Auto(N)$ with $(\psi,r)$ on the isotropy subgroup   such that $\sigma = (\psi,r)\cdot\gamma$ is also a magnetic trajectory satisfying $\sigma(0)=e$ and  $\sigma'(0)=x_0e_1+y_0e_2+z_0e_3$ with {\bf  $x_0\geq 0$}. So we can restrict ourselves to this case. 



By making use of the action of  $\Iso(H_3)\times \RR^*$  on the set of left-invariant Lorentz forces, as in Proposition \ref{orbits}, we may choose $U=e_1$ and  system \eqref{magnetic-Heis} reduces  to:
\begin{equation}\label{magnetic-type2HeisSimp2}
	\left\{ \begin{array}{rcl}
		x''(t)+( x(t)+z_0+\rho) (y'(t)+1) & = &  \rho\\
		y''(t)-( x(t)+z_0+\rho) x'(t)  & = & 0 \\
		z'(t)+\frac{1}{2}(x'(t)y(t)-x(t)(y'(t)+2)) & = & z_0
	\end{array} \right.
\end{equation}
with initial conditions $x(0)=y(0)=z(0)=0$ and $x'(0)=x_0, y'(0)=y_0, z'(0)=z_0$.

The first observation is that to solve Equations \eqref{magnetic-type2HeisSimp2} one needs to solve only an ordinary differential equation in one variable.

\begin{lem}\label{1solution}
	Let $F_{e_1, \rho}$ be a left-invariant Lorentz force on $(H_3,\la\, ,\,\ra)$. There is a one-to-one correspondence between:
	$$\left\{\begin{array}{c}
		\gamma(t)=\exp(x(t)e_1+ y(t) e_2 + z(t)e_3):\\ \mbox{solutions of Equation \eqref{magnetic-type2HeisSimp2}}\\
		\gamma(0)=e \mbox{ and } \gamma'(0)=x_0e_1+y_0e_2+z_0e_3
	\end{array}\right\}  \longleftrightarrow     \left\{\begin{array}{c}
		x(t): \mbox{solution of Equation }\\
		x''(t)+h'(x(t)) h(x(t))=\rho \\
		x(0)=0 \mbox{ and } x'(0)=x_0
	\end{array}\right\} $$ 
	where	$h(x)= \frac{x^2}{2}+(z_0+\rho)x+y_0+1$.

	Moreover, the functions $y(t)$ and $z(t)$ are determined by $x(t)$ as follows
	\begin{equation}\label{eqmagneticHeisy}y(t)=\int_{0}^{t}\left(\frac{x(s)^2}{2}+(z_0+\rho)x(s)+y_0\right)ds,\end{equation} 
	\begin{equation}\label{eqmagneticHeisz}z(t)=-\frac{1}{2}x(t)y(t)-(z_0+\rho)y(t)-x'(t)+x_0.\end{equation}		
\end{lem}

\begin{proof} From the first and third equations in \eqref{magnetic-type2HeisSimp2} we have that:
	\begin{align*}
		z'(t)
		=-\frac{1}{2}x'(t)y(t)-\frac{1}{2}x(t)y'(t) -(z_0+\rho)y'(t)-x''(t), 
	\end{align*}
	which by integration implies \eqref{eqmagneticHeisz}.  
	
	Analogously, from the second equation of the system \eqref{magnetic-type2HeisSimp2} and the fact that $y'(0)=y_0$ one has
	\begin{equation*}
		y''(t)=\left(\frac{x(t)^2}{2}+(z_0+\rho) x(t)\right)'\Rightarrow y'(t)=\frac{x(t)^2}{2}+(z_0+\rho)x(t)+y_0.
	\end{equation*}
	Implying the formula \eqref{eqmagneticHeisy} since $y(0)=0$.
	
	Using the expression of $y'(t)$ above in the first equation of the system \eqref{magnetic-type2HeisSimp2} we get
	\begin{equation}\label{eqmagneticHeisx}
		x''(t)+( x(t)+z_0+\rho) \left(\frac{ x(t)^2}{2}+(z_0+\rho) x(t)+y_0+1\right) =  \rho.
	\end{equation}
	
	Conversely, once one has a solution $x(t)$ of Equation \eqref{eqmagneticHeisx} define $y(t)$ and $z(t)$ by Formulas \eqref{eqmagneticHeisy} and \eqref{eqmagneticHeisz}, respectively, to get a magnetic curve $\gamma(t)=\exp(x(t)e_1+ y(t) e_2 + z(t)e_3)$ solution of the system \eqref{magnetic-type2HeisSimp2} with $\gamma(0)=e$ and $\gamma'(0)=x_0e_1+y_0e_2+z_0e_3$. 
\end{proof}

\begin{rem} There is a analogous result to the Lemma above for a left-invariant Lorentz force $F_{0,\rho}$. 	
	The corresponding equations were  written in Equation  \eqref{magnetic-exact}.  
		
For the proof of this assertion note that from the second equation in system \eqref{magnetic-exact} one has $y'(t)=y_0+(z_0+\rho)x(t)$  and  by replacing in the first equation, one gets  $$x''(t)+(z_0+\rho)[y_0+(z_0+\rho)x(t)]=0.$$ 
	Conversely given a solution $x(t)$ of the previous equation define $y(t)$ and $z(t)$ by
	$$y(t)=\int_{0}^{t} \left((z_0+\rho)x(s)+y_0 \right) ds$$
	$$z(t)=-\frac{1}{2}x(t)y(t)+\int_{0}^{t}\left((z_0+\rho)x(s)^2+y_0 x(s)+z_0\right)ds$$
	to get a magnetic curve $\gamma(t)=exp(x(t)e_1+y(t)e_2+z(t)e_3)$ solution of the system \eqref{magnetic-exact}.
	
	
\end{rem}

By the Cauchy-Kovalevskaya theorem Equation \eqref{eqmagneticHeisx} has a unique real analytic solution with $x(0)=0$ and $x'(0)=x_0$. Thus, the solution is $x(t)\equiv 0$ if and only if $x_0=0$ and $(z_0+\rho)(y_0+1)=\rho$.

From now on we assume that
\begin{equation}\label{nonconst}
	x_0\neq 0 \mbox{ or }  (z_0+\rho)(y_0+1)\neq\rho,
\end{equation} so that $x(t)$ is not constant in any interval.  Then we work out Equation \eqref{eqmagneticHeisx} in the following way
\begin{align*}
	x''(t)+h'(x(t))h(x(t))=\rho \\
	\mbox{ equivalently to }\quad   x'(t)x''(t)+x'(t)h'(x(t))h(x(t))=\rho x'(t)\\
	\Leftrightarrow \qquad   \left(x'(t)^2 +h(x(t))^2\right)'=2\rho x'(t)\\
	\mbox{ and this says that } \quad  x'(t)^2 +h(x(t))^2 = 2\rho x(t) + x_0^2+(y_0+1)^2.
\end{align*}
The function $h$ was defined in Lemma \ref{1solution}. 
So, we get that $x'$ may solve the equation:
\begin{equation}\label{Eqonev}
	x'(t)^2 =\Vert V_0+e_2\Vert ^2-h(x(t))^2 + 2\rho x(t).
\end{equation}

\medskip



In the following paragraphs we shall see that the polynomial $\Vert V_0+e_2\Vert^2-h(x)^2 + 2\rho x$ plays a role for the description of the solutions of Equation \eqref{magnetic-type2HeisSimp2}. 

In fact, to solve Equation \eqref{Eqonev}, we assume that $x_0>0$, so  $x'(t)>0$ in some neighborhood of $0$. Thus,  one has:
$$\frac{x'(t)}{\sqrt{\Vert V_0+e_2\Vert^2-h(x(t))^2 + 2\rho x(t)}}  =1$$ 
implying that
$$\int_{0}^{x(t)}\frac{d\tau}{\sqrt{\Vert V_0+e_2 \Vert^2-h(\tau)^2 + 2\rho \tau}}  =t.$$
Hence the solution $x(t)$ is the inverse of the elliptic integral
\begin{equation}\label{ellipticint}
	\mathcal{E}(x)=
	\int_{z_0+\rho}^{x+z_0+\rho}\frac{d\eta}{\sqrt{\Vert V_0+e_2\Vert^2-\left(\frac{\eta^2}{2}+y_0+1-\frac{(z_0+\rho)^2}{2}\right)^2+2\rho (\eta-z_0-\rho)}}.
\end{equation}
where $\eta=\tau+z_0+\rho$.

We have to analyze the quartic polynomial
\begin{equation}\label{polinQ}
	\begin{split}
		P(\eta)&= \Vert V_0+e_2\Vert^2-\left(\frac{\eta^2}{2}+y_0+1-\frac{(z_0+\rho)^2}{2}\right)^2+2\rho (\eta-z_0-\rho)\\
		&=-\frac{1}{4}\left(\eta^4 + 2p_0 \eta^2 -8\rho \eta + q_0  \right) 
	\end{split}
\end{equation}  
where
\begin{equation}
	p_0=2(y_0+1) - (z_0+\rho)^2  \ \mbox{ and } \  q_0=p_0^2+8\rho(z_0+\rho) - 4\Vert V_0+e_2\Vert^2.
\end{equation}


Firstly, observe that

\begin{itemize}
	\item $P$ has negative principal coefficient ($c_4=-\frac{1}{4}$),
	\item $P(z_0+\rho)=x_0^2\geq 0$ and
	\item  $P'(z_0+\rho)=2(\rho-(y_0+1)(z_0+\rho))$.
\end{itemize} 
By the assumption \eqref{nonconst}, the values $P(z_0+\rho)$ and $P'(z_0+\rho)$ are not simultaneously null. We conclude that $P$ has at least two distinct real roots and $P$ is non-negative in a closed interval of positive length containing $z_0+\rho$. Thus $\mathcal{E}(x)$ is a real function on some interval containing $0$.  

To determine the nature of the roots we have to consider the discriminant of $P$:
\begin{equation}\label{delta}
	\Delta = q_0 p_0^4-8\rho^2 p_0^3-432\rho^4+72 \rho^2 q_0 p_0-2 q_0^2 p_0^2+q_0^3.
\end{equation}
We have three possible cases:
\begin{enumerate}[(a)]
	\item if $\Delta <0$, then  $P$ has two distinct real roots and two complex conjugate non-real roots,
	\item if $\Delta >0$, then $P$ has four distinct real roots (cannot be all non-real),
	\item if $\Delta =0$, then $P$ has a root of multiplicity greater than one and all roots are real.
\end{enumerate}

Suppose that $\Delta\neq0$. Let $r_1,$ $r_4$ be the lower and greatest real roots of $P$, respectively, and $r_2$, $r_3$ the other roots in $\mathbb{C}$, that can be real or not. By Vieta's formulas we have that:
\begin{eqnarray}
	r_2+r_3&=&-(r_1+r_4)\label{Viete1}\\
	r_2 r_3&=& 2p_0 + r_1^2 + r_4^2 + r_1 r_4\label{Viete2}\\
	(r_1+r_4)(2p_0 + r_1^2 + r_4^2) &=& 8\rho. \label{Viete3} 	
\end{eqnarray}

Observe that $2p_0+r_1^2+r_4^2\neq0$. In fact, since $2p_0+r_1^2+r_4^2=0$ one obtains $\rho=0$ and from the expression of $P$ in Equation \eqref{polinQ} it is direct to see that $r_1^2=r_4^2=z_0^2+\Vert V_0+e_2\Vert - 2(y_0+1)=-p_0$. Then $\Vert V_0+e_2\Vert=0$, i.e. $x_0=0$ and $y_0=-1$. But this correspond to the trivial solution and we have supposed that $x_0\neq 0$ or $z_0(y_0+1)\neq 0$.

From \eqref{Viete1} and \eqref{Viete2} we have that:
\begin{equation*}
	r_2=-\frac{r_1+r_4}{2}-\sqrt{\frac{(r_4-r_1)^2}{4}-2p_0-r_1^2-r_4^2}  \qquad r_3=-\frac{r_1+r_4}{2}+\sqrt{\frac{(r_4-r_1)^2}{4}-2p_0-r_1^2-r_4^2}.
\end{equation*}


In particular,
\begin{equation}\label{reldisneg}
	\Delta<0 \iff 2p_0+r_1^2+r_4^2 > \frac{(r_4-r_1)^2}{4}>0.
\end{equation}

If $\Delta=0$  we take $r_1=r_4=r$. The previous relations \eqref{Viete1} and \eqref{Viete2} also hold and we get $r_2=-r-\sqrt{-2(p_0+r^2)}$ and $r_3=-r+\sqrt{-2(p_0+r^2)}$. All the roots must be real in this case so $p_0+r^2<0$.   




\begin{caseof}
	
	\case{$\Delta <0$} In this case $r_1 < r_4\in\RR$ and $r_2 =\overline{r_3}\in\CC-\RR$.
	Using Equation 259.00 of \cite{BF} and inverting $\mathcal{E}$ we get that the solution of Equation \eqref{eqmagneticHeisx} is given by
	\begin{equation}\label{solcase1}
		x(t)=\frac{\left(r_1 \delta_4-r_4 \delta_1\right)\mathrm{cn}\left(\frac{\sqrt{\delta_1 \delta_4}}{2}t + C_1,k\right)+r_1 \delta_4 + r_4 \delta_1 }{\left(\delta_4-\delta_1\right)\mathrm{cn}\left(\frac{\sqrt{\delta_1 \delta_4}}{2}t+C_1,k\right) + \delta_1 + \delta_4} -z_0-\rho,
	\end{equation}
	where $\mathrm{cn}$ is the cosine amplitude, Jacobi's elliptic function, with 
	
	$\delta_1=|r_1-r_2|=\sqrt{2p_0+2r_1^2+(r_1+r_4)^2}$, and 
	$\delta_4=|r_4-r_2|=\sqrt{2p_0+2r_4^2+(r_1+r_4)^2}$,

	$k^2=\frac{(r_4-r_1)^2 - (\delta_4-\delta_1)^2}{4\delta_1 \delta_4}$, while  
	$C_1=\mathrm{cn}^{-1}\left(\frac{(r_4-z_0-\rho)\delta_1 - (z_0+\rho-r_1)\delta_4}{(r_4-z_0-\rho)\delta_1 + (z_0+\rho-r_1)\delta_4},k\right)$. Observe that the constant $C_1$ is well defined since $r_1\leq z_0+\rho\leq r_4$.


	By considering the triangle inequality corresponding to $r_1$, $r_4$ and $r_2$ on the complex plane, more precisely $\delta_1+\delta_4\geq r_4-r_1$, one gets that $0<k<1$. Recall that $\mathrm{cn}(x,k)$ is periodic of period $4K(k)$ for $0<k<1$ and its image is $[-1,1]$, where $K(k)$ is complete elliptic integral of the first kind defined as
	\begin{equation}
		K(k)=\int_{0}^{\pi/2} \frac{1}{\sqrt{1-k^2 sin^2 \theta}}d\theta.
	\end{equation} 	
	Thus, the solution $x(t)$ is periodic of period $\frac{8K(k)}{\sqrt{\delta_1 \delta_4}}$ and its image is $[r_1-z_0-\rho,r_4-z_0-\rho]$.
	
	See more details about properties of elliptic functions in \cite{Di}.


	\case{$\Delta >0$} We have that $r_1 < r_2 < r_3 < r_4\in\RR$. By considering the region where $P$ is positive, we have two possibilities: either  $r_1 \leq z_0+\rho \leq r_2$ or $r_3 \leq z_0+\rho \leq r_4$. 
	
	If $r_1 \leq z_0+\rho \leq r_2$, we use Equation 252.00 of \cite{BF} to get the solution
	\begin{equation}\label{solcase21}	
		x(t)=r_4-\frac{r_4-r_1}{1+\frac{r_2-r_1}{r_4-r_2}\mathrm{sn}^2\left(\frac{\sqrt{(r_4-r_2)(r_3-r_1)}}{4}	t + C_{21},k_1\right)}-z_0-\rho,
	\end{equation}
	where $k_1=\sqrt{\frac{(r_4-r_3)(r_2-r_1)}{(r_4-r_2)(r_3-r_1)}}$ and
	$C_{21}=\mathrm{sn}^{-1}\left(\sqrt{\frac{(r_4-r_2)(z_0+\rho-r_1)}{(r_2-r_1)(r_4 -z_0-\rho)}},k_1\right)$. Observe that $k_1 <1$. 
	
	Since the function $\mathrm{sn}^2(x,k_1)$ is periodic of period $2K(k_1)$ and image $[0,1]$, then the function $x(t)$ is periodic of period $\frac{8K(k_1)}{\sqrt{(r_4-r_2)(r_3-r_1)}}$ and image $[r_1-z_0-\rho, r_2-z_0-\rho]$.
	
	We can rewrite solution \eqref{solcase21} exactly as \eqref{solcase1} using elliptic identities
	where \\$\delta_1=\sqrt{(r_2-r_1)(r_3-r_1)}$,  $\delta_4=\sqrt{(r_4-r_3)(r_4-r_2)}$, $k^2=\frac{(r_4-r_1)^2-(\delta_4-\delta_1)^2}{4\delta_1 \delta_4}$ and $C_1=2\sqrt{k_1}C_{21}$. The difference with the previous case is that now $k>1$. In fact by the AM-GM inequality we have that $\delta_1<\frac{r_2+r_3}{2}-r_1$ and $\delta_4<r_4-\frac{r_2+r_3}{2}$, which gives  $(\delta_1-\delta_4)^2+4\delta_1\delta_4=(\delta_1+\delta_4)^2 <(r_4-r_1)^2$.

	If $r_3 \leq z_0+\rho \leq r_4$ we have that: 
	\begin{equation}\label{solcase22}
		x(t)=r_1+\frac{r_4-r_1}{1+\frac{r_4-r_3}{r_3-r_1}\mathrm{sn}^2\left(\frac{\sqrt{(r_4-r_2)(r_3-r_1)}}{4}	t - C_{31},k_1\right)}-z_0-\rho,
	\end{equation}	
	which can be written also 
	as:
	\begin{equation}\label{solcase221}
		x(t)=\frac{(r_1 \delta_4 - r_4 \delta_1)\,\mathrm{cn}\left(\frac{\sqrt{\delta_1 \delta_4}}{2}t-C_3,k\right)-r_1 \delta_4- r_4\delta_1}{(\delta_4-\delta_1)\,\mathrm{cn}\left(\frac{\sqrt{\delta_1 \delta_4}}{2}t-C_3,k\right)-\delta_1-\delta_4} -z_0-\rho,
	\end{equation}
	where $k_1$ and $k$ are the same as before,
	with 	$C_{31}=\mathrm{sn}^{-1}\left(\sqrt{\frac{(r_3-r_1)(r_4-z_0-\rho)}{(r_4-r_3)( z_0+\rho-r_1)}},k_1\right)$ and $C_3=2\sqrt{k_1}C_{31}.$ 
	This solution is also periodic with the same period $\frac{8K(k_1)}{\sqrt{(r_4-r_2)(r_3-r_1)}}$ but its image is $[r_3-z_0-\rho, r_4-z_0-\rho]$.

	\case{$\Delta =0$} Let $r\in\RR$ be the root of $P$ with multiplicity two and consider $\mu=\frac{p_0 +3r^2}{2}=\frac{1}{4}(r-r_2)(r-r_3)$. See the paragraph after Equation \eqref{reldisneg}. 
	\begin{enumerate}[(a)]
		\item If $\mu=\frac{p_0 +3r^2}{2}>0$ the solution of \eqref{eqmagneticHeisx} is given by 
		\begin{equation}\label{sol3a}
			x(t)=\frac{-2\mu }{r+\sqrt{r^2-\mu}\cos\left(\sqrt{\mu}\,t-C_4\right)}+r-z_0-\rho,
		\end{equation}
		where
		$C_4=\arccos\left(\frac{1}{\sqrt{r^2-\mu}}\left(-\frac{2\mu}{z_0+\rho-r}-r\right)\right)$.

		In this case the solution is also periodic with period $\frac{2\pi}{\sqrt{\mu}}$ and its image is the interval 
		$\left[-r- 2\sqrt{-\frac{p_0+r^2}{2}}-z_0-\rho, -r+ 2\sqrt{-\frac{p_0+r^2}{2}}-z_0-\rho\right]$.

		\item If $\mu=\frac{p_0 +3r^2}{2}<0$. There are two possibilities, either   $r_2=-r-\sqrt{-2(p_0+r^2)}\leq z_0+\rho<r$ or $r<z_0+\rho\leq r_3= -r+\sqrt{-2(p_0+r^2)}$.
		
		If $r<z_0+\rho\leq -r+\sqrt{-2(p_0+r^2)}$, the solution is
		\begin{equation}\label{sol3b1}
			x(t)=\frac{- 2\mu
			}{r+\sqrt{r^2-\mu} \cosh(\sqrt{-\mu}\,t+C_5)}+r-z_0-\rho,
		\end{equation}
		where
		$C_5= \arccosh\left(\frac{1}{\sqrt{r^2-\mu}}\left(-\frac{2\mu}{z_0+\rho-r}-r\right)\right)$.

		If $-r-\sqrt{-2(p_0+r^2)}\leq z_0+\rho<r$, we have the solution 
		\begin{equation}\label{sol3b2}
			x(t)=\frac{- 2 \mu
			}{r-\sqrt{r^2-\mu} \cosh(\sqrt{-\mu}\,t+C_6)}+r-z_0-\rho,
		\end{equation}
		where
		$C_6= -\arccosh\left(\frac{1}{\sqrt{r^2-\mu}}\left(\frac{2\mu}{z_0+\rho-r}+r\right)\right)$.
		
		In both cases the solutions are not periodic but their images are respectively 
		
		$\left( r-z_0-\rho,-r+\sqrt{-2(p_0+r^2)}-z_0-\rho\right]$ and $\left[-r-\sqrt{-2(p_0+r^2)}-z_0-\rho, r-z_0-\rho\right)$.

		\medskip
		
		\item Finally in the case that $\mu=\frac{p_0 +3r^2}{2}=0$, the solution of Equation \eqref{eqmagneticHeisx} is given by
		\begin{equation}\label{sol3c}
			x(t)=\frac{ -4r}{1+r^2(t+C_7)^2}+r-z_0-\rho,
		\end{equation}
		where $C_7=\frac{1}{r}\sqrt{\frac{3r+z_0+\rho}{r-z_0-\rho}}$. This function is not periodic and its image is the interval $(r-z_0-\rho, -3r-z_0-\rho]$ if $\rho>0$ and $[-3r-z_0-\rho, r-z_0-\rho)$ if $\rho<0$.

	\end{enumerate}

\end{caseof}


\begin{rem}
	Observe that the solutions for the case $\Delta=0$ are written in terms of $r$, $\rho$ and $z_0$. In these cases we can express $r$ explicitly in terms of the initial conditions. In fact, in cases  (3)(a) and (3)(b) where $p_0^2+3q_0\neq 0$ the following formulas hold:
	$$r=\frac{2\rho(p_0^2+3q_0)}{p_0^3-p_0q_0+36\rho^2} \quad \text{ and } \quad  \mu=6\frac{9\rho^2-p_0q_0}{p_0^2+3q_0} + \frac{p_0}{2}.$$ 
	In case (3)(c) one has that $p_0^2+3q_0=0$  and $r=-\sqrt[3]{\rho}$.
\end{rem}

\begin{rem}	
	Taking into account that $\mathrm{cn(x,0)}=\cos(x)$ and $\mathrm{cn(x,1)}=\sech(x)$ we see that all solutions (except the last one) can be written as in Equation \eqref{solcase1} where $r_1$ and $r_4$ are real distinct roots of $P$, $r_2$ and $r_3$ the remaining roots (real or complex),  $\delta_1=\sqrt{(r_2-r_1)(r_3-r_1)}$, $\delta_4=\sqrt{(r_4-r_2)(r_4-r_3)}$ and $k^2=\frac{(r_4-r_1)^2 - (\delta_4-\delta_1)^2}{4\delta_1\delta_4}$.
\end{rem}





\begin{thm}\label{thmsolutionmagHeisII}
	The solution of Equation \eqref{eqmagneticHeisx} with initial conditions $x(0)=0$ and $x'(0)=x_0$ is summarized below:
	\begin{enumerate}[(i)]
		\item If $x_0= 0$ and $(y_0+1)(z_0+\rho)= \rho$, then $x(t)=0$ for every $t\in\RR$.
		
		\item If $x_0^2 + (\rho-(y_0+1)(z_0+\rho))^2\neq 0$ and $x_0\geq0$, the equation's number for  $x(t)$ is given in the second column of Table \ref{table:1},		where one has:
		\begin{itemize}
			\item $p_0=2(y_0+1)-(z_0+\rho)^2$,
			\item $q_0=p_0^2+8\rho(z_0+\rho)-4(x_0^2+(y_0+1)^2)$,
			\item $\Delta=q_0 p_0^4-8\rho^2 p_0^3-432\rho^4+72 \rho^2 q_0 p_0-2 q_0^2 p_0^2+q_0^3$,
			\item $\mu=\frac{9\rho^2-p_0q_0}{p_0^2+3q_0} + \frac{p_0}{12}$ and $r=\frac{2\rho(p_0^2+3q_0)}{p_0^3-p_0q_0+36\rho^2}$.
			\item If $\Delta<0$, the real numbers $r_1<r_4$ and the non-real complex numbers $r_2$ and $\overline{r_2}$ are the roots of the polynomial $P$ in Equation \eqref{polinQ} and $k=\sqrt{\frac{(r_4-r_1)^2 - (|r_4-r_2|-|r_1-r_2|)^2}{4|r_4-r_2||r_1-r_2|}}$.
			\item If $\Delta>0$, $r_1<r_2<r_3<r_4$ are the real roots of $P$ and $k_1=\sqrt{\frac{(r_4-r_3)(r_2-r_1)}{(r_4-r_2)(r_3-r_1)}}$.
		\end{itemize} 
		
	\end{enumerate}
	
\end{thm}

\renewcommand{\arraystretch}{1.9}

\begin{table}
	\begin{tabular}{|c|c|c|c| }
		\hline			
		Condition & $x(t)$ & Period & Image \\
		\hline\hline
		$\Delta < 0$ & 
		\eqref{solcase1} & $\frac{8K(k)}{\sqrt{|r_1-r_2|| r_4-r_2|}}$ & $[r_1-z_0-\rho,r_2-z_0-\rho]$ \\ \hline
		$\Delta > 0$, $r_1\leq z_0+\rho\leq r_2$ & \eqref{solcase21} & $\frac{8K(k)}{\sqrt{(r_4-r_2)(r_3-r_1)}}$ & $[r_1-z_0-\rho, r_2-z_0-\rho]$ \\ \hline
		$\Delta > 0$, $r_3\leq z_0+\rho\leq r_4$ & \eqref{solcase22} & $\frac{8K(k)}{\sqrt{(r_4-r_2)(r_3-r_1)}}$ & $[r_3-z_0-\rho, r_4-z_0-\rho]$\\ \hline
		$\Delta = 0$, $\mu>0$ & \eqref{sol3a} & $\frac{2\pi}{\sqrt{\mu}}$ & {\tiny $\left[-r- \sqrt{-2(p_0+r^2)}-z_0-\rho, -r+ \sqrt{-2(p_0+r^2)}-z_0-\rho\right]$} \\ \hline
		$\Delta = 0$, $\mu<0$, $z_0+\rho> r$ & \eqref{sol3b1} & non & $\left(r-z_0-\rho,-r+ \sqrt{-2(p_0+r^2)}-z_0-\rho\right]$ \\ \hline
		$\Delta = 0$, $\mu<0$, $z_0+\rho <r$ & \eqref{sol3b2} & non & $\left[-r- \sqrt{-2(p_0+r^2)}-z_0-\rho,r-z_0-\rho\right)$ \\ \hline
		$\Delta = 0$, $p_0^2+3q_0 =0$  & \eqref{sol3c} & non  & \begin{tabular}{c}
			$\left(-\sqrt[3]{\rho}-z_0-\rho,3\sqrt[3]{\rho}-z_0-\rho\right]$ if $\rho>0$\\ $\left[3\sqrt[3]{\rho}-z_0-\rho,-\sqrt[3]{\rho}-z_0-\rho\right)$ if $\rho<0$
		\end{tabular}\\ \hline
	\end{tabular}
	\caption{ }
	\label{table:1}
\end{table}

\begin{rem}
	By Lemma \ref{lem2} and Proposition \ref{orbits}, the magnetic trajectory satisfying $\sigma(0)=e\in H_3$ and $\sigma'(0)=x_0 e_1 + y_0 e_2 + z_0 e_3$ with $x_0<0$ is given by:
	
	\smallskip
	
	$\sigma(t)=((S,-1)\cdot\gamma)(t)=exp(x(-t)e_1-y(-t)e_2-z(-t)e_3)$,
	
	\smallskip
	
	where $\gamma(t)$ is the magnetic trajectory given in the previous theorem with $\gamma'(0)=-x_0 e_1 + y_0 e_2 + z_0 e_3$.
\end{rem}

\begin{rem}
	Every magnetic trajectory is defined for all $t\in\RR$.
	Also a magnetic trajectory $\gamma(t)$ corresponding to a Lorentz force $F_{e_1,\rho}$ will be a one-parameter subgroup if and only if $x_0= 0$ and $(y_0+1)(z_0+\rho)= \rho$. This trajectory is central, this means $\gamma(t)\in Z(H_3)$ for all $t$, if and only if $y_0=z_0=0$, that is when $\gamma(t)= e$ for all $t$.    
\end{rem}

\begin{rem}
	Fixed $\rho\neq 0$, every situation in Table \ref{table:1} is possible under suitable initial conditions. We assume $\rho>0$ ( the case $\rho<0$ is analogous). First we observe that $\Delta$ is a polynomial on $z_0$ of degree $5$ if $\Vert V_0+e_2\Vert\neq 0$, or $3$ if $\Vert V_0+e_2\Vert=0$, so it takes all real values for fixed $x_0$ and $y_0$. More explicitly, 
	\begin{itemize}
		\item if  $x_0=0$, $y_0=0$ and $z_0=-\rho$, thus $\Delta<0$;
		\item if  $x_0=0$, $y_0<-\frac{3}{2}\sqrt[3]{2}\rho^{2/3}-1$ and $z_0=-\rho$, then $\Delta>0$ and $z_0+\rho =0=r_3$;
		\item if  $x_0=0$, $y_0=-2\rho-\frac{3}{4}$ and $z_0=-\rho-1$ then $\Delta>0$ and $z_0+\rho=-1=r_1$ or $r_2$.
	\end{itemize}
	
	On the other side,  examples with $\Delta=0$ can be obtained as follows: 
	\begin{itemize}
		\item if $(x_0,y_0,z_0)=(0, 3\rho^{2/3}-1, 3\rho^{1/3}-\rho)$ then $\mu=0$;
		\item if $(x_0,y_0,z_0)=(0, -\frac{3}{2}\sqrt[3]{2}\rho^{2/3}-1, -\rho)$ then $\mu>0$;
		\item if $(x_0,y_0,z_0)=(0, 2(2\rho)^{2/3}-1, \frac{5}{2}(2\rho)^{1/3}-\rho)$ then $\mu<0$ and $z_0+\rho > r$;
		\item and if $(x_0,y_0,z_0)=(0, 3\rho^{2/3}-1, -\frac{15}{4} \rho^{1/3}-\rho)$ then  $\mu<0$ and $z_0+\rho < r$.
	\end{itemize}	 
\end{rem}


	\subsection{Magnetic trajectories that are geodesics}
	Now we shall study the question of which magnetic trajectories are geodesics. From Equation 
	\eqref{mageq}, a magnetic trajectory $\gamma: I \to M$ that  is geodesic must satisfy 
	$$F\gamma'(t)=0, $$
	for $t$ in an interval in $\RR$.  Thus, if the Lorentz map $F$  is non-singular, there is no non-trivial magnetic trajectory being geodesic. 
	
	So, the equations for geodesics on any 2-step nilpotent Lie group $(N,\la\,,\, \ra)$ as above follow (see Equation \eqref{magnetic-2step})
		\begin{equation}\label{geodesic}
		\left\{ \begin{array}{rcl}
			V''(t)-j(Z'(t)+\frac12 [V'(t), V(t)])V'(t) & = & 0\\
			Z''(t)+\frac12 [V''(t),V(t)] & = & 0, 
		\end{array} \right.
	\end{equation}
	and  we  also have
	$$F(V'(t) + Z'(t) + \frac12 [V'(t),V(t)])\equiv 0.$$
\begin{example} Let $(N,\la\,,\,\ra)$ be a 2-step nilpotent Lie group equipped with a left-invariant metric and left-invariant Lorentz force $F$. 
	\begin{itemize}
		\item If $Z_0\in \ker F \cap \zz$ then $\exp(tZ_0)$ is a magnetic trajectory which is also a geodesic.
		
		\item If $V_0\in \ker F \cap \vv$ then $\exp(t V_0)$ is a magnetic trajectory which is also a geodesic.
	\end{itemize}
\end{example}

\begin{lem}\label{lemMagGeod} Let $(N, \la\,,\,\ra)$ be a 2-step nilpotent Lie group equipped with a left-invariant metric and a left-invariant Lorentz force $F$. 
\begin{enumerate}[(i)]
	\item If $\ker F\cap \vv=\{0\}$ then the only magnetic trajectories which are geodesics have the form
	$$\gamma(t)=\exp(t(V_0+Z_0))\qquad \mbox{ for  } Z_0\in \ker F\cap \zz \mbox{ and } V_0\in \ker j(Z_0).$$
	
	\item  If $\ker F= \RR V_0$ for some $V_0\in\vv$ then the only magnetic trajectories which are geodesics have the form
	$$\gamma(t)=\exp(t \nu V_0)\qquad \mbox{ for  some } \nu\in\RR.$$
\end{enumerate}

	\end{lem}
			\begin{proof} (i) From the second equation in \eqref{geodesic} we have $Z'(t)+\frac12[V'(t),V(t)]\equiv Z_0$. On the other side , it holds $0=F(V'(t)+Z_0)$ which implies $FV''(t)=0$. So, one gets $V'(t)\equiv V_0$, and $V(t)=V_0t$. The first equation in \eqref{geodesic} says that $J(Z_0)V_0=0$, and from the second equation in \eqref{geodesic} we have $Z(t)=tZ_0$ with $Z_0\in \ker F$. Thus $\gamma(t)=\exp(t(V_0+Z_0))$.

		(ii) As in (i), $Z'(t)+\frac12 [V'(t),V(t)]\equiv Z_0$. By hypothesis, \( F(V'(t) + Z_0) = 0 \) which implies \( Z_0 = 0 \) and \( V'(t) = a(t)V_0 \) for some smooth function \( a : I \to \mathbb{R} \). Consequently, \( [V'(t), V(t)] = 0 \) and  so \( Z'(t) \equiv 0 \). It follows that \( Z(t) \equiv 0 \). Using the first equation in~\eqref{geodesic}, we  obtain \( V''(t) = a'(t)V_0 = 0 \). Therefore,
		\[
		V(t) = t \nu V_0, \qquad \text{for some } \nu \in \mathbb{R} \setminus \{0\}.
		\]
		
			\end{proof}
			
			\begin{cor}\label{mageo}
				On the Heisenberg Lie group $(H_3, \la\,,\,\ra)$ the only magnetic trajectories which are also geodesics are
				the one-parameter subgroups
				\begin{itemize}
					\item $\exp(tZ_0)$ where $Z_0\in \zz$ for $F_{0,\rho}$ with $\rho\neq 0$, 
					\item $\exp(tV_0)$ where $V_0\in \vv$ for a Lorentz force $ F_{\nu JV_0,0}$ with $\nu\neq 0$.
				\end{itemize}
			\end{cor}
			
			\begin{proof} Observe that \( \ker F_{U,\rho} = \mathrm{span}\{JU - \rho e_3\} \). Thus, if $\rho=0$, then $\ker F_{U,0}\subset \vv$. If $\rho\neq 0$ then $\ker F_{U,\rho}\cap \vv=\{0\}$. In both cases the previous lemma is applied. Note that \( \ker j(e_3) = \{0\} \).  
						\end{proof}
		\begin{rem}
			Clearly, the fact that only one-parameter subgroups are both geodesics and magnetic trajectories is a specific feature of \( H_3 \). Indeed, in the trivial extension $H_3 \times \mathbb{R}^n$ with $n \ge 2 $, one can construct a  nontrivial Lorentz force $F$ such that the restriction to $\hh_3$ is $F_{|_{\hh_3}} \equiv 0$. Consequently, every geodesic of $H_3 \subset H_3 \times \mathbb{R}^n$ is a magnetic trajectory for such an $F$.
			
		\end{rem}	
			\section{The inverse problem of the calculus of variations}
			
			In this section  we deal with the inverse problem of the calculus of variations. Explicitly, for the magnetic equations on $H_3$, the question is	to determine whether the solution curves can be identified with the totality of extremals of some variational problem
			$$\int  L(t,y_i,y_i')\, dt = \min,$$
			and in the  affirmative case to find possible functions $L$. This is solved in this section.

			The question now is to prove the existence of  a Lagrangian  for Equations \eqref{magnetic-Heis}. And moreover to determine precise conditions for a given function to became one. 
			For such Lagrangian, the corresponding  Euler-Lagrange equations  (Eq. \eqref{ELeq} below)  are equivalent to the system \eqref{magnetic-Heis}.
			
			Assume that there exists a differentiable function $L:\RR\times \RR^n\times \RR^n \to \RR$, which will play the role of Lagrangian. 
			Let $A$ be the action functional given by
			$$A(u)=\int_0^T  L(t,y_i,y_i')\, dt,$$ 
			where $u$ is a curve on $\RR^n$, $u:[0,T]\to \RR^n$, $u(t)=(y_1(t), \hdots y_n(t))$. The principle of least action says that, given an initial state $x_0$ and a final state $x_1$ in $\RR^n$, the trajectory that the system determined by $L$ will actually follow must be a minimizer of the action functional $S$ satisfying the boundary conditions $u(0) = x_0$, $u(T) = x_1$.
			The critical points (and hence minimizers) of $A$ must satisfy the Euler–Lagrange equations:
			\begin{equation}\label{ELeq}
				\frac{d}{dt} \frac{\partial L}{\partial y_i'} - \frac{\partial L}{\partial y_i} = 0 \qquad \mbox{ for all }i=1, \hdots, n.
			\end{equation}
			
		Recall that a 2-form $\Omega$ is called {\em exact} whenever there exists a 1-form such that $d\theta=\Omega$. 
		
		Note that a basis of left-invariant 1-forms is given by
		$$e^1=dx\quad e^2 = dy \quad e^3= dz + \frac12 (y dx - x dy), $$
		and one gets $de^1 = 0 = de^2$, while $d e^3 =  e^1 \wedge e^2$. 
	
	More generally, 	let $\theta$ be a differential 1-form on $H_3$: 
	$$\theta =f_1 dx + f_2 dy+ f_3 dz, \mbox{ where }f_i\in C^{\infty}(\RR^3).$$ 
Thus,  its differential follows:
	$$d\theta=(\frac{\partial f_2}{\partial x} - \frac{\partial f_1}{\partial y}) dx\wedge dy + (\frac{\partial f_3}{\partial x} - \frac{\partial f_1}{\partial z}) dx\wedge dz + (\frac{\partial f_3}{\partial y} - \frac{\partial f_2}{\partial z}) dy\wedge dz.$$
		On the other hand a Lorentz force $F$ with matrix as in \eqref{matrixF}  can be identified with the following closed  left-invariant 2-form
		$$\begin{array}{rcl}
			\omega_F & = &\rho e^1 \wedge e^2 + \beta e^1\wedge e^3 + \alpha e^2\wedge e^3\\			
			& = & \rho dx \wedge dy + \beta dx \wedge (dz - \frac12 x dy) + \alpha dy\wedge(dz+\frac12 y dx).
		\end{array}
		$$
		
		\begin{lem}\label{closed2areexact} Any left-invariant closed 2-form $\omega_F=\la F\cdot,\cdot\ra$ on $H_3$ verifies $d\theta =\omega_F$, if and only if the 1-form $\theta$  given by
			$$
			\theta =f_1 dx + f_2 dy+ f_3 dz$$
			with $f_i\in C^{\infty}(\RR^3)$ for $i=1,2,3$, satisfy
			\begin{equation}\label{exactforms}
				\frac{\partial f_2}{\partial x} - \frac{\partial f_1}{\partial y}  =  \rho - \frac{\beta}{2} x   - \frac{\alpha}{2} y \qquad 
				\frac{\partial f_3}{\partial x} - \frac{\partial f_1}{\partial z}  =  \beta \qquad
				\frac{\partial f_3}{\partial y} - \frac{\partial f_2}{\partial z}  = \alpha.
			\end{equation}
			 In particular, for 
			$\theta =(-\frac{\rho}{2}y +\frac{\beta}2 xy) dx +(\frac{\rho}{2}x -\frac{\alpha}2 xy) dy +(\rho+\beta x +\alpha y) dz$, one gets
			$$d\theta = \rho e^1 \wedge e^2 + \beta e^1\wedge e^3 + \alpha e^2\wedge e^3 .$$
			Note that the 1-form $\theta$ is not left-invariant in general. The left-invariant situation occurs whenever  $\alpha=\beta=0$.
		\end{lem}
		
		\begin{rem} 
		Clearly, one can find 1-forms $\theta_i$ such that
			\begin{equation}\label{closed-cond}
				\begin{array}{rcl}
					d\theta_1 & = & e^1\wedge e^2 = dx \wedge dy\\
					d \theta_2 & = & e^1\wedge e^3 = dx \wedge dz - \frac12 x\, dx \wedge dy \\
					d \theta_3 & = & e^2\wedge e^3 = dy \wedge dz - \frac12 y\, dx \wedge dy\\
				\end{array}
			\end{equation}
		\end{rem}
			
				The aim now is to find Lagrangians, so that  magnetic equations  for any left-invariant Lorentz force on $H_3$ correspond to Euler-Lagrange equations for such Lagrangian $L$. 
			
			\begin{example}\label{exaexac}
				As noticed in \cite{EGM}, there is a Lagrangian for left-invariant exact 2-forms. In fact, take:
				$$L(t,x,y,z,\dot{x},\dot{y},\dot{z})= \frac12 \left( \dot{x}^2+\dot{y}^2+(\dot{z}+\frac12(\dot{x}y-x\dot{y}))^2\right)-\rho\left(\dot{z}+\frac12(\dot{x}y-x\dot{y})\right).$$
			\end{example}

		For an exact 2-form $\Omega=d\theta$
		with $\theta$ a smooth 1-form on $M$, the magnetic flow can be obtained as the Euler-Lagrange flow of the Lagrangian
		$$L(p,v)=\frac12 \la v, v\ra_p - \theta_p(v), \qquad (p,v)\in TM.$$

			\begin{thm} \label{Lagrangian}
				The system of magnetic equations on the Heisenberg Lie group $H_3$ is a Lagrangian system. In fact,  Lagrangian functions can be written as
				$$	L= T - \theta, \quad \mbox{ with} \quad T= \frac12\left(\dot{x}^2+\dot{y}^2+(\dot{z}+\frac12(\dot{x}y-x\dot{y}))^2\right) $$ and $\theta$ satisfying Equation \eqref{exactforms}.
				An explicit Lagrangian for the system of equations \eqref{magnetic-Heis} is given by
				$$\begin{array}{rcl}
					L(t,x,y,z,\dot{x},\dot{y},\dot{z}) & = & \frac12 \left( \dot{x}^2+\dot{y}^2+(\dot{z}+\frac12(\dot{x}y-x\dot{y}))^2\right) \\
					& &  +(\frac{\rho}{2} y -\frac{\beta}{2}xy) \dot{x} + (-\frac{\rho}{2} x +\frac{\alpha}{2}xy) \dot{y} - (\beta x +\alpha y)\dot{z}.
				\end{array}$$
			\end{thm}
			\begin{proof}
				We work out the explicit Lagrangian. By computing one gets
				\begin{equation}\label{ecu1}
					0=\frac{d}{dt} \frac{\partial L}{\partial \dot{z}}  - \frac{\partial L}{\partial z}=z''+\frac12(x''y-xy'')-(\beta x'+\alpha y'),
				\end{equation}
				which coincides with the third equation in System \eqref{magnetic-Heis}. Also, 
				\begin{equation}\label{ecu2}
					\begin{array}{rcl}
						0=\frac{d}{dt} \frac{\partial L}{\partial \dot{x}} - \frac{\partial L}{\partial x} & = & (1+\frac{y^2}4)x''-\frac14 xyy''+\frac12 yz''+\\
						&&+ z'(y'+\beta) + y'\left(\frac12(x'y-xy')+\rho-\frac{\beta}{2} x-\frac{\alpha}{2} y\right),
					\end{array}
				\end{equation}
				
				\begin{equation}\label{ecu3}
					\begin{array}{rcl}
						0=\frac{d}{dt} \frac{\partial L}{\partial \dot{y}} - \frac{\partial L}{\partial y} & = & (1+\frac{x^2}4)y''-\frac14 xyx''+\frac12 xz''+\\
						&&+ z'(-x'+\alpha) - x'\left(\frac12(x'y-xy')+\rho-\frac{\beta}{2} x-\frac{\alpha}{2} y\right).
					\end{array}
				\end{equation}
				
				From Equation \eqref{ecu1} one has the next equalities
				$$\frac12 y (z''+\frac12(x''y-xy'')-\alpha y')=\frac{\beta}2 x'y$$
				$$\frac12 x (z''+\frac12(x''y-xy'')-\alpha y')=\frac{\alpha}2 xy'.$$
				By replacing in Equations \eqref{ecu2} and \eqref{ecu3}, respectively, and  following usual computations one obtains the two first equations of the system \eqref{magnetic-Heis}. 
			\end{proof}
		
		\begin{rem}
			Let $M=\Gamma\backslash H_3$ be a compact manifold, where $\Gamma < H_3$ is a discrete subgroup of $H_3$. Fix a point $p\in M$ and let $\pi:H_3 \to M$ be the usual projection. Then the Lagrangian above defines a local Lagrangian around $\pi(p)$. 
			
In fact, note that for any left-invariant Lorentz force we get the same magnetic equations on $M$ and on $H_3$.  Thus we obtain the same local solutions and local Lagrangians.

On the other hand,{\it There is no global Lagrangian on $M$.} 

In fact, assume we have a Lagrangian $L$ of $TM$. Locally we have the same magnetic equations as in $H_3$. Moreover any possible Lagrangian differs from another one by a  closed 1-form. 

Thus, $L - T$ is a 1-form $\theta$ such that $d\theta$ is the invariant magnetic field globally defined on $M$ and induced by the Lorentz force $F$. But by Nomizu's Theorem the cohomology of $M$ coincides with the cohomology of the Lie algebra. And  $\dim H^2(\hh_3)= 2$.

		\end{rem}
			
			\

\end{document}